\documentclass{amsart}
\usepackage{color}
\usepackage{amsmath}
\usepackage[all]{xy}
  \usepackage{graphicx}
  \usepackage{epsf}
  \usepackage{tikz}
  \usepackage{caption}
  \usepackage{enumitem}
  
\usetikzlibrary{arrows,shapes}
  
\usetikzlibrary{positioning,automata}
  
\usepackage{hyperref}
\usepackage[active]{srcltx}

\newtheorem{Thm}{Theorem}[section]

\def\KK{\ensuremath{\mathsf {Kar}}}
\def\pv#1{\ensuremath{{\mathsf{#1}}}}

\begin{document}

\title{Symbolic dynamics and semigroup theory}
\author{Alfredo Costa}
%\date{2018}
\maketitle
%\begin{multicols}{2}

  A major motivation for the development of semigroup theory
  was, and still is, its applications
  to the study of formal languages. Therefore,
  it is not surprising that the correspondence
  $\mathcal X\mapsto B(\mathcal X)$,
  associating to each symbolic dynamical system $\mathcal X$ the formal
  language $B(\mathcal X)$ of its blocks, entails
  a connection between symbolic dynamics and semigroup theory.
  In this article we survey some developments
  on this connection,
  since when it was noticed
  in an article by Almeida, published in the CIM bulletin, in 2003~\cite{Almeida:2003b}.
  
\section{Symbolic dynamics}

A \emph{topological dynamical system} is
a pair $(X,T)$ consisting of
a topological space $X$ and a continuous self-map $T\colon X\to X$.
It is useful to think of $X$ as
 representing a sort of ``space'', where each point $x$ is moved
 to $T(x)$ when a unit of time
 has passed.
A morphism between two topological dynamical systems
$(X_1,T_1)$ and $(X_2,T_2)$ is a continuous map $\varphi\colon X_1\to X_2$
such that $\varphi\circ T_1=T_2\circ\varphi$.
In this way, topological dynamical systems form a category, if
we take the identity on $X$
as the local identity at $(X,T)$.
In this category, an isomorphism is called a \emph{conjugacy},
and isomorphic objects are said to be \emph{conjugate}.

We focus on a special class of topological
dynamical systems, the symbolic ones.
Their applications in the study of
 general topological dynamical systems
 frequently stem from the following procedure:
 use symbols to mark a finite number of regions of the underlying space,
 and register, with a string of those symbols, the regions visited by a orbit.
 In the next paragraph we give a brief formal introduction to symbolic systems. For a more developed introduction, we refer to the book~\cite{Lind&Marcus:1995}. Also, the book review~\cite{Sunic:2014} is an excellent short introduction to the field and its ramifications.
 
Consider
a finite nonempty set $A$, whose elements are called \emph{symbols}, and
the set $A^{\mathbb Z}$ of sequences $(x_i)_{i\in\mathbb Z}$
of symbols of $A$ indexed by $\mathbb Z$. One should think
of an element $x=(x_i)_{i\in\mathbb Z}$ of $A^{\mathbb Z}$ as
a biinfinite string
$\ldots x_{-3}x_{-2}x_{-1}.x_0x_1x_2x_3\ldots$,
with the dot marking the reference position.
A \emph{block} of $x$ is a finite string appearing in
$x$: a finite sequence
of the form $x_{k}x_{k+1}\ldots x_{k+\ell}$,
with $k\in\mathbb Z$ and $\ell\geq 0$, 
also denoted $x_{[k,k+\ell]}$. Of special
relevance are the central blocks $x_{[-k,k]}$,
as one endows $A^{\mathbb Z}$
with the topology induced by the metric
$d(x,y)=2^{-r(x,y)}$
such that $r(x,y)$
is the minimum $k\geq 0$
for which $x_{[-k,k]}\neq y_{[-k,k]}$.
Hence, two elements of $A^{\mathbb Z}$ are ``close''
if they have a ``long'' common central block. 
The \emph{shift mapping} $\sigma\colon A^{\mathbb Z}\to A^{\mathbb Z}$, defined by $\sigma(x)=(x_{i+1})_{i\in\mathbb Z}$,
shifts the dot to the right.
A \emph{symbolic dynamical system}, also called
\emph{subshift},
is a pair $(\mathcal X,\sigma_{\mathcal X})$ formed
by a nonempty closed subspace $\mathcal X$ of $A^{\mathbb Z}$,
for some $A$, such that $\sigma(\mathcal X)=\mathcal X$, and
by the restriction $\sigma_{\mathcal X}$ of $\sigma$ to~$\mathcal X$.
As it is clear what self-map is considered,
$(\mathcal X,\sigma_{\mathcal X})$ is
identified with~$\mathcal X$.
When $\mathcal X=A^{\mathbb Z}$, the system is a \emph{full~shift}.
The \emph{sliding block code} from the subshift
$\mathcal X\subseteq A^{\mathbb Z}$ to the subshift $\mathcal Y\subseteq B^{\mathbb Z}$, with \emph{block map} $\Phi\colon A^{m+n+1}\to B$, \emph{memory} $m$ and \emph{anticipation}~$n$,
is the map $\varphi\colon \mathcal X\to \mathcal Y$
defined by $\varphi(x)=(\Phi(x_{[i-m,i+n]}))_{i\in \mathbb Z}$.
It follows from
the definition of the metric on a full shift
that the morphisms between subshifts
are precisely the sliding block codes.

\section{Formal languages}

Given a set $A$ of symbols, the set of finite nonempty strings
of elements of $A$ is denoted by $A^+$. In the jargon
of formal languages, $A$ is said to be
an \emph{alphabet}, the elements of $A$ and those of $A^+$
are respectively called \emph{letters} and \emph{words}, and
the subsets of $A^+$ are \emph{languages}. Moreover,
$A^+$ is viewed as a semigroup for the operation ``concatenation of words''.
For example, in $\{a,b\}^+$, the  product of
$aba$ and $bab$
is $ababab$.
In fact, $A^+$ is the \emph{free semigroup generated by $A$},
since, for every semigroup $S$,
every mapping $A\to S$ extends uniquely to a homomorphism $A^+\to S$.
Concerning semigroups, formal languages,
and their interplay, we give~\cite{Almeida:2005bshort} as a source of detailed references and as a very convenient guide, since, in this sort of introductory text,
connections with symbolic dynamics are also highlighted.

If $\mathcal X$ is a subshift of $A^{\mathbb Z}$,
we let $B(\mathcal X)$ be the language over the alphabet
$A$ such that $u\in B(\mathcal X)$
if and only if $u$ is a block of some element $x$
of $\mathcal X$.
As a concrete example, consider the
subshift $\mathcal E$,
known as the \emph{even shift},
formed by the biinfinite sequences of $a$'s and $b$'s that
have no odd number of $b$'s between two consecutive
$a$'s, that is, the biinfinite paths in the following labeled graph:
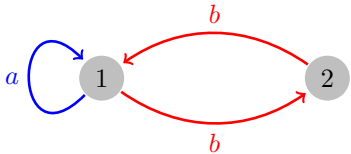
\tikzstyle{state}=[circle,fill=black!25,minimum size=17pt,inner sep=0pt]
\tikzset{every loop/.style={min distance=10mm,in=240,out=-60,looseness=5}}  
\begin{center}
    \begin{tikzpicture}[shorten >=1pt, node distance=2.5cm and 3cm, on grid,initial text=,semithick,line width=1pt]
  \node[state]   (1)   {$1$};
  \node[state]   (2) [right=of 1]   {$2$};
  \path[->]  (1)   edge  [bend right=35,red] node [below] {$b$} (2)
  (2)   edge  [bend right=35,red]  node [above] {$b$} (1)
  (1)   edge [out=225,in=135,looseness=8,blue]  node [auto] {$a$} (1);
\end{tikzpicture}
\end{center}

A language $L$ is \emph{factorial} if, for each
$u\in L$, every factor of $u$ belongs
to $L$. A factorial language over $A$ is \emph{prolongable}
if $u\in L$ implies $aub\in L$ for some $a,b\in A$.
It is easy to see that the languages of the form $B(\mathcal X)$,
with $\mathcal X$ a subshift of $A^{\mathbb Z}$, are precisely
the factorial prolongable languages over $A$. Moreover,
the correspondence $\mathcal X\mapsto B(\mathcal X)$
is a bijection between subshifts and
factorial prolongable languages.
Moreover, one has $\mathcal X\subseteq\mathcal Y$ if and only if
$B(\mathcal X)\subseteq B(\mathcal Y)$.
In view of this bijection, symbolic dynamics may be regarded as a subject of formal language theory.

Semigroups appear in the study of formal languages via the concept of \emph{recognition}.
In the labeled graph of the figure above, letters
$a$ and $b$ may be seen as
the binary relations
$a=\{(1,1)\}$ and $b=\{(1,2),(2,1)\}$. Let $S(\mathcal E)$ be the semigroup
of binary relations, on the vertices $1$ and $2$, generated by
$a$ and~$b$. For example, $ab$ is the binary relation $\{(1,2)\}$.
The words in $B(\mathcal E)$ are precisely the words
that in $S(\mathcal E)$ are not the empty relation~$\emptyset$.
Formally, given a semigroup homomorphism $\varphi\colon A^+\to S$,
a language $L\subseteq A^+$ is \emph{recognized}
by $\varphi$ if $L=\varphi^{-1}(P)$ for some subset
$P$ of~$S$.
Note that $B(\mathcal E)$
is recognized by the homomorphism $\varphi\colon \{a,b\}^+\to S(\mathcal E)$
such that $\varphi(a)=\{(1,1)\}$
and $\varphi(b)=\{(1,2),(2,1)\}$, since
$B(\mathcal E)=\varphi^{-1}(S(\mathcal E)\,\backslash\,\{\emptyset\})$.

A language over $A$ is \emph{recognized by
the semigroup $S$} when recognized by a homomorphism from $A^+$ into $S$.
It is said to be \emph{recognizable} if it is recognized
by a \underline{finite} semigroup. Recognizable languages
constitute one of the main classes of languages,
as they describe ``finite-like'' properties
of words, captured by finite devices.
Frequently the devices are finite automata,
which are labeled graphs with
a distinguished set of initial vertices
and final vertices. These devices recognize
the words labeling the paths from the initial to the final vertices.
Recognition by a finite automaton is the same
as recognition by a finite semigroup, because
in fact an automaton
may be seen as a semigroup with generators acting on its vertices.

Another reason why recognizable languages
matter is Kleene's theorem (1956)~\cite{Kleene:1956}, stating that the recognizable languages of $A^+$,
with $A$ finite, are precisely
the \emph{rational} languages of $A^+$, that is, the languages
which can be obtained from subsets of $A$ by applying
finitely many times the Boolean operations,
concatenation of languages, and the operation
that associates to each nonempty language $L$ the subsemigroup
$L^+$ of $A^+$ generated by $L$.
The rational languages
obtainable using only the first two of these three sets of operations,
the \emph{plus-free} languages\footnote{Actually, Sch\"utzenberger's result is usually formulated in terms of finite aperiodic monoids and
  languages admitting the empty word, with \emph{star-free} languages in place
of \emph{plus-free} languages.}, 
are precisely the languages recognized by finite \emph{aperiodic} semigroups~\cite{Schutzenberger:1965}.
This characterization, due to Sch\"utzenberger and dated from 1965, is one
of the first important applications of semigroups
to languages (for the reader unfamiliar with the concept: a semigroup is aperiodic if all its subgroups, i.e., subsemigroups
that have a group structure, are trivial).
Eilenberg, later on (1976), provided the framework
for several results in the spirit of that of Sch\"utzenberger on aperiodic semigroups,
by establishing a natural correspondence between
\emph{pseudovarieties of semigroups} (classes of finite semigroups
closed under taking homomorphic images, subsemigroups and finitary products)
and the types of classes of languages recognized by their semigroups, called
\emph{varietes of languages}~\cite{Eilenberg:1976}.

\section{Classification of subshifts}

The correspondence $\mathcal X\mapsto B(\mathcal X)$
provides ways of
classifying subshifts in
special classes
with ``static'' definitions
in terms of $B(\mathcal X)$ that, from a semigroup theorist viewpoint,
may be more convenient than the alternative
definitions of a more ``dynamical'' flavor.

As a first example, consider the~\emph{irreducible} subshifts:
these are the subshifts $\mathcal X$
such that, for every $u,v\in B(\mathcal X)$,
one has $uwv\in B(\mathcal X)$ for some word $w$.
The dynamical characterization
is that a subshift is irreducible when it has a dense forward orbit.

In the same spirit, a subshift $\mathcal X$ is \emph{minimal} 
(for the inclusion)
if and only if
$B(\mathcal X)$ is \emph{uniformly recurrent},
the latter meaning that
for every $u\in B(\mathcal X)$,
there is a natural number $N_u$
such that $u$ is a factor of every word of $B(\mathcal X)$ of
length $N_u$. Note that uniform recurrence implies
irreducibility. A procedure for building minimal subshifts,
with a semigroup-theoretic flavor that was useful for
getting results mentioned in the final section, is as follows.
Consider a \emph{primitive substitution} $\varphi\colon A^+\to A^+$,
i.e., a semigroup endomorphism $\varphi$ of $A^+$
such that every letter of $A$ appears in $\varphi^n(a)$,
for all $a\in A$ and all sufficiently large~$n$:
if $\varphi$ is not the identity in an one-letter alphabet,
then the language of factors of words of the form $\varphi^k(a)$, with
$k\geq 1$ and $a\in A$, is factorial and prolongable,
thus defining a subshift $\mathcal X_\varphi$,
and in fact this subshift is minimal.

A subshift $\mathcal X$ is \emph{sofic} when
$B(\mathcal X)$ is recognizable.
Hence, the even subshift is sofic.
Sofic and minimal subshifts are arguably the most important
big realms of subshifts, with only periodic subshifts in the intersection.
Every subshift $\mathcal X$ of $A^{\mathbb Z}$
is characterized by a set~$F$ of \emph{forbidden blocks},
a language $F\subseteq A^+$ such that
$x\in \mathcal X$ if and only if
no element of $F$ is a block of $x$.
We write $\mathcal X=\mathcal X_F$ for such a set $F$.
It turns out that $\mathcal X$ is sofic
if and only if $F$ can be chosen to be rational.
A subshift $\mathcal X$ is of \emph{finite type} if there is a finite set
of forbidden blocks $F$ such that $\mathcal X=\mathcal X_F$.
The class of finite type subshifts is closed under conjugacy
and is contained in the class of sofic subshifts.
The inclusion is strict: the even subshift is not
a finite type subshift.

The most important open problem in
symbolic dynamics consists in classifying (irreducible) finite type subshifts up to conjugacy.
A related problem is the classification of (irreducible) sofic subshifts up to \emph{flow equivalence}.
 In few words,
 two subshifts are flow equivalent when they have equivalent
 mapping tori,
 a description that is somewhat technical,
 when made precise. Next is an alternative characterization (from~\cite{Parry&Sullivan:1975}), more
prone to a semigroup theoretical approach.
Take $\alpha\in A$
and  a letter $\diamond$ not in $A$.
Consider
the homomorphism $E_\alpha\colon A^+\to (A\cup\{\diamond\})^+$
that replaces $\alpha$ by $\alpha\diamond$ and leaves the remaining
letters of $A$ unchanged.
The symbol expansion of a subshift $\mathcal X\subseteq A^{\mathbb Z}$ with respect to $\alpha$
is the subshift
whose blocks are factors of words in
$E_\alpha(B(\mathcal X))$.
Flow equivalence
is
the least equivalence
relation
between subshifts
that contains the conjugacy relation and the symbol expansions.
A symbol expansion
on $\alpha$ represents a time dilation when reading $\alpha$ in
a biinfinite string, thus flow equivalence
preserves ``shapes'' of orbits, but not in a ``rigid'' way.
Finite type subshifts have been completely classified
up to flow equivalence~\cite{Franks:1984}. The strictly sofic
case remains open.
In~\cite{Boyle&Carlsen&Eilers:2018}
one finds recent  developments.

\section{The Karoubi envelope of a subshift}

Let $L$ be a language over $A$.
Two words $u$
and $v$ of $A^+$ are \emph{syntactically equivalent} in $L$
if they share the contexts in which they appear in
words of $L$.
Formally, the \emph{syntactic congruence} $\equiv_L$
is defined by
$u\equiv_L v$
if and only if
the equivalence
$xuv\in L\Leftrightarrow xvy\in L$
holds, for all (possibly empty)
words $x,y$ over $A$.
 The quotient $S(L)=A^+/{\equiv_L}$
is the \emph{syntactic semigroup} of $L$.
The quotient homomorphism
$\eta_L\colon A^+\to A^+/{\equiv_L}$
is minimal among the onto homomorphisms
recognizing $L$: if the onto homomorphism $\varphi\colon A^+\to S$
recognizes $L$, then there is a unique onto
homomorphism $\theta\colon S\to S(L)$
such that the diagram
\begin{equation*}
  \xymatrix{
    A^+
    \ar[r]^\varphi
    \ar[rd]_{\eta_L}
    &S\ar[d]^{\theta}\\
    &S(L)
  }
\end{equation*}
commutes.
In particular,
$L$ is recognizable if and only if $S(L)$ is finite.
More generally, $L$ is recognized by a semigroup of a pseudovariety $\pv V$
if and only if $S(L)$ belongs to $\pv V$.
For example,
a language is plus-free if and only if $S(L)$ is an aperiodic semigroup,
in view of Sch\"utzenberger's characterization of plus-free languages.
Since $S(L)$ is computable if $L$ is adequately described (e.g, by an automaton), this gives an algorithm to decide if
a rational language is plus-free.
This example
illustrates why syntactic semigroups
and pseudovarieties are important for studying rational languages.

Let $S$ be a semigroup, and denote by $E(S)$ the set of idempotents
of $S$. The~\emph{Karoubi envelope} of~$S$
is the small category $\KK(S)$ such that
\begin{itemize}
\item the set of objects is $E(S)$;
\item an arrow from $e$ to $f$ is a triple $(e,s,f)$
  such $s\in S$ and $s=esf$;
\item composition of consecutive arrows
  is given by $(e,s,f)(f,t,g)=(e,st,g)$ (we compose
  on the opposite direction adopted by category theorists);
\item the unit at vertex $e$ is $(e,e,e)$.
\end{itemize}
This construction found an application
in finite semigroup theory in the \emph{Delay Theorem}~\cite{Tilson:1987}.
Avoiding details, this result concerns a certain correspondence
$\pv V\mapsto \pv V'$ between
semigroup pseudovarieties,
with one of the formulations of the Delay Theorem
stating that a finite semigroup $S$
belongs to $\pv V'$
if and only if 
$\KK(S)$ is the quotient
of a category admitting a faithful functor into a monoid in~$\pv V$.
Interestingly, the variety of languages
corresponding in Eilenberg's sense to $\pv V'$ 
is, roughly speaking, determined
by the inverse images
of languages recognized by semigroups of~$\pv V$
via block maps of sliding block codes.
Hence, it is natural to
relate the Karoubi envelope with subshifts. This was done
in the paper~\cite{ACosta&Steinberg:2016},
of which we highlight some results in the next paragraphs.

The \emph{syntactic semigroup $S(\mathcal X)$ of a subshift~$\mathcal X$} is the syntactic semigroup
of $B(\mathcal X)$.
One finds this object
in some papers~\cite{Jonoska:1996a,Jonoska:1998,Beal&Fiorenzi&Perrin:2005b,Beal&Fiorenzi&Perrin:2005a,Costa:2006,Costa:2006b,Chaubard&Costa:2008},
namely for (strictly) sofic subshifts.
Several invariants
encoded in $S(\mathcal X)$
were deduced.
The \emph{Karoubi envelope of $\mathcal X$},
denoted $\KK(\mathcal X)$, is
the Karoubi envelope of
$S(\mathcal X)$.
Conjugate subshifts
do not need to have isomorphic syntactic semigroups, but
the Karoubi envelope of a subshift is invariant in the sense of the following result from~\cite{ACosta&Steinberg:2016}.

\begin{Thm}\label{t:splittingisinvflow}
If $\mathcal X$ and $\mathcal Y$ are flow equivalent
subshifts, then the categories $\KK(\mathcal X)$ and $\KK(\mathcal Y)$ are equivalent.
\end{Thm}

For some classes of subshifts,
the Karoubi envelope is of no use.
For example, all irreducible finite type subshifts
have equivalent Karoubi envelopes.
But in the strictly sofic case, the Karoubi envelope
does bring  meaningful information,
as testified by several examples given in~\cite{ACosta&Steinberg:2016}.
We already mentioned the previous existence in the literature
of several (flow equivalence) invariants encoded
in $S(\mathcal X)$.
It turns out that the Karoubi envelope
is the best possible syntactic invariant for flow equivalence of sofic subshifts:
indeed, the main result in~\cite{ACosta&Steinberg:2016},
which we formulate precisely below, states that all syntactic invariants
of flow equivalence of sofic subshifts are encoded in the Karoubi envelope.
First, it is convenient to formalize what a syntactic flow invariant is.
An equivalence relation $\vartheta$ on the class of sofic subshifts is:
 an \emph{invariant of flow equivalence} if
  $\mathcal X \mathrel{\vartheta} \mathcal Y$ whenever $\mathcal X$
  and $\mathcal Y$ are flow equivalent;
 a \emph{syntactic invariant} if
  $\mathcal X \mathrel{\vartheta} \mathcal Y$ whenever $S(\mathcal X)$ and $S(\mathcal Y)$ are isomorphic;
 a \emph{syntactic invariant of flow equivalence} if it satisfies the two former properties.

\begin{Thm}\label{t:master-syntactic-invariant}
If $\vartheta$ is a syntactic invariant of flow equivalence
of sofic subshifts and $\mathcal X$ and $\mathcal Y$ are sofic
shifts such that $\KK(\mathcal X)$ is equivalent to $\KK(\mathcal Y)$, then
$\mathcal X\mathrel{\vartheta}\mathcal Y$.
\end{Thm}

Outside the sofic realm, the Karoubi envelope
was successfully applied in~\cite{ACosta&Steinberg:2016}
to what is arguably an almost complete classification
of the \emph{Markov-Dyck subshifts},
a class of subshifts introduced by Krieger~\cite{Krieger:2000}.
Loosely speaking, a Markov-Dyck subshift~$D_G$ is formed by biinfinite
strings of several types
of parentheses, subject to the
usual parenthetic rules, and to additional restrictions
defined by a graph $G$. The edges of $G$ are
the opening parentheses, and consecutive opening parentheses
appearing in an element of $D_G$ correspond to consecutive edges, with a natural
symmetric rule for closing parentheses also holding.
Flow invariance of $\KK(D_G)$, together with a characterization
of $S(D_G)$, implicit in~\cite{Jones&Lawson:2014},
gives the following result (a different and independent proof appears in~\cite{Krieger:2015}).

\begin{Thm}
  Let $G$ and $H$ be finite graphs.
  If each vertex of $G$ or of $H$
  has out-degree not equal to one and in-degree at least one,
  then $D_G$
  and $D_H$ are flow equivalent
  if and only if $G$ and $H$ are isomorphic.
\end{Thm}

\section{Free profinite semigroups}

We already looked at the importance
of (pseudovarieties~of) finite semigroups in the study of (varieties~of)
rational languages.
It is well known that
free algebras (e.g., free groups, free Abelian groups, free semigroups, etc.)
are crucial for the study of varieties
of algebras, but for pseudovarieties,
a difficulty arises: there is no universal object within the category of \emph{finite} semigroups.
To cope with this difficulty, an approach successfully
followed by semigroup theorists,
since the 1980's, was to enlarge the class of finite semigroups, by considering
profinite semigroups. We pause to define the latter,
giving~\cite{Almeida:2005bshort} as a supporting reference.

A \emph{profinite semigroup}
is a compact semigroup (i.e., one with a compact Hausdorff
topology for which the semigroup operation is continuous)
that is \emph{residually finite}, in the sense that every
pair of distinct elements $s,t$ of $S$ admits a continuous homomorphism $\varphi$ from $S$ onto a finite
semigroup $F$ such that
$\varphi(s)\neq\varphi(t)$, where finite semigroups
get the discrete topology.

Assuming $A$ is finite,
consider in $A^+$ the metric
$d(u,v)=2^{-r(u,v)}$ such that $r(u,v)$ is the least possible size
of the image of a homomorphism
$\psi\colon  A^+\to S$ satisfying $\psi(u)\neq \varphi(v)$.
The  completion~$\widehat{A^+}$ of $A^+$ with respect to $d$ is a profinite
semigroup.
Moreover, each map
$\varphi\colon A\to S$ from $A$ into
a profinite semigroup $S$ has a unique extension to a continuous homomorphism $\widehat\varphi\colon \widehat{A^+}\to S$. Hence, $\widehat{A^+}$ is
the \emph{free profinite semigroup generated by~$A$}.
The next theorem gives a glimpse of
why free profinite  semigroups matter~\cite{Almeida:1994a}.
This theorem identifies the free profinite semigroup as the Stone dual of the Boolean algebra of recognizable languages.

\begin{Thm}
  The recognizable languages of $A^+$, are
  the traces in $A^+$ of the clopen subsets of $\widehat{A^+}$:
  if $L\subseteq A^+$ is recognizable,
  then $\overline{L}$ is clopen in $\widehat{A^+}$, and, conversely, if $K$ is clopen in $\widehat{A^+}$,
  then $K\cap A^+$ is recognizable.
\end{Thm}

The elements of $\widehat{A^+}$
constitute a sort of generalization of the words in $A^+$,
and for that reason they are often named \emph{pseudowords}.
The elements in
$\widehat{A^+}\,\textbackslash\, A^+$ are the \emph{infinite} pseudowords over $A$.
While the algebraic-topological structure of $A^+$ is poor,
that of $\widehat{A^+}$ is very rich: for example,
$A^+$ has no subgroups, while $\widehat{A^+}$
contains all finitely generated free profinite
groups when $|A|\geq 2$, and actually many more groups~\cite{Rhodes&Steinberg:2008}.
The structure of $\widehat{A^+}$
is nowadays less mysterious than it was fifteen years ago,
symbolic dynamics having been very useful for achieving that.
Our goal until the end of the text is to give examples
of such utility. 

Most connections between symbolic dynamics and free profinite semigroups
developed over Almeida's idea of considering,
for each subshift $\mathcal X$ of $A^{\mathbb Z}$,
the topological closure $\overline{B(\mathcal X)}$ of $B(\mathcal X)$
in $\widehat {A^+}$~\cite{Almeida:2003b,Almeida:2005bshort}.

In a semigroup $S$, the quasi-order $\leq_{\mathcal J}$
is defined by $s\leq_{\mathcal J}t$ if and only if
$t$ is a factor of $s$. The
equivalence relation on $S$ induced by $\leq_{\mathcal J}$
is denoted by $\mathcal J$.
By standard compactness arguments, when $\mathcal X$
is an irreducible subshift there is a $\leq_{\mathcal J}$-minimum $\mathcal J$-class
of $\widehat{A^+}$ among the $\mathcal J$-classes contained in $\overline{B(\mathcal X)}$ (equivalently, intersecting $\overline{B(\mathcal X)}$),
as explained in~\cite{Costa&Steinberg:2011}.
This $\mathcal J$-class
is denoted $J(\mathcal X)$.
The proof of the existence of $J(\mathcal X)$ also entails
that it is a \emph{regular} $\mathcal J$-class,
that is, one that contains idempotents.
One has $J(\mathcal X)=J(\mathcal Y)$
if and only if $\mathcal X=\mathcal Y$, and so
$J(\mathcal X)$ contains all information about $\mathcal X$.
Something more holds: one
has $\mathcal X\subseteq \mathcal Y$
if and only if $J(\mathcal Y)\leq_{\mathcal J}J(\mathcal X)$.
For the next statement, have in mind
that an infinite pseudoword $u$ of $\widehat{A^+}$
is a \emph{$\leq_{\mathcal J}$-maximal infinite
pseudoword} if every factor of $u$
either belongs to $A^+$
or is $\mathcal J$-equivalent to $u$.

\begin{Thm}[\cite{Almeida:2005c}]
  An element $u$ of $\widehat{A^+}$ is a $\mathcal J$-maximal infinite
  pseudoword if and only if $u\in J(\mathcal X)$
  for some minimal subshift $\mathcal X$ of $A^{\mathbb Z}$.
\end{Thm}

The next theorem states that, in a natural sense, $\widehat{A^+}$ is very ``large'' and very ``high'' (a weaker version appears in~\cite{Costa:2001a}, with
a harder proof).
Its proof is a good example
of the potential of symbolic dynamics in the study of free profinite semigroups.
A \emph{regular} pseudoword is one that is $\mathcal J$-equivalent to
an idempotent.

\begin{Thm}\label{t:ALCO-chains-antichains}
  Let $A$ be an alphabet with at least two letters.
  For the relation $<_{\mathcal J}$ in
  $\widehat{A^+}$, there are both chains and anti-chains with $2^{\aleph_0}$
  regular elements.
\end{Thm}

\begin{proof}
  On the one hand, $A^{\mathbb Z}$ contains $2^{\aleph_0}$
  minimal subshifts (cf.~\cite[Chapter~2]{Lothaire:2001}),
  and minimal subshifts clearly form an anti-chain
  for the inclusion.
  On the other hand,
  $A^{\mathbb Z}$ contains a chain of $2^{\aleph_0}$ irreducible
subshifts~\cite[Section 7.3]{Walters:1982}. Hence,
the theorem follows immediately from the
equivalence \mbox{$\mathcal X\subseteq \mathcal Y
\Leftrightarrow
J(\mathcal Y)\leq_{\mathcal J}J(\mathcal X)$}
for irreducible
subshifts.
\end{proof}

Since $J(\mathcal X)$
is regular, it contains
a maximal subgroup, which is a profinite group for the induced topology.
Because all maximal subgroups in
a regular $\mathcal J$-class are isomorphic,
we may consider the abstract profinite maximal subgroup $G(\mathcal X)$ of $J(\mathcal X)$.
The group $G(\mathcal X)$
was called in~\cite{Almeida&ACosta:2013}
the \emph{Sch\"utzenberger group of $\mathcal X$}.
This group is a conjugacy invariant (see~\cite{Costa:2006} for a proof).
We collect other facts about~$G(\mathcal X)$.

\begin{itemize}
\itemsep=0 cm
\item In~\cite{Almeida:2005c} it was shown that $G(\mathcal X)$
  is a free profinite group of rank $k$
  if $\mathcal X$ is
  a subshift over a $k$-letter alphabet that
  belongs to an extensively studied
  class of minimal subshifts, called Arnoux-Rauzy subshifts.
  On the other hand,
  also in~\cite{Almeida:2005c},
  it was shown that the substitution $\varphi$
  defined by $\varphi(a)=ab$
  and $\varphi(b)=a^3b$
  is such that $G(\mathcal X_\varphi)$
  is not a free profinite group. This was the first example
  of a non-free maximal subgroup of a free profinite semigroup.
  More generally, profinite presentations for
  $G(\mathcal X_\psi)$
  were obtained in~\cite{Almeida&ACosta:2013}, for all primitive substitutions
  $\psi$.
    
\item If~$\mathcal X$ is a nonperiodic
 irreducible sofic subshift, then $G(\mathcal X)$ is a free profinite group
 of rank~$\aleph_0$~\cite{Costa&Steinberg:2011}.
 
  \item A sort of geometrical interpretation
    for $G(\mathcal X)$ was obtained in~\cite{Almeida&ACosta:2016b}, when
    $\mathcal X$ is minimal.
    It was shown that $G(\mathcal X)$
    is an inverse limit
    of the profinite completions of the fundamental groups
    of a certain sequence of finite graphs.
    The $n$-th graph in this sequence captures
    information about the blocks of $\mathcal X$ with length $2n+1$.
\end{itemize}

While free profinite semigroups
are interesting \emph{per se},
it is
worthy mentioning that some of the achievements on the Sch\"utzenberger group
of a minimal subshift were used
in the technical report~\cite{Kyriakoglou&Perrin:2017}
to obtain results on code theory,
whose statement  may appear
to have nothing to do with profinite semigroups.
 These results were incorporated and further developed in~\cite{Almeida&ACosta&Kyriakoglou&Perrin:2018}.

\section*{Acknowledgments}

Work partially supported by the Centre for Mathematics of the
    University of Coimbra -- UID/MAT/00324/2013, funded by the Portuguese
    Government through FCT/MCTES and co-funded by the European Regional 
    Development Fund through the Partnership Agreement PT2020.
 The author is indebted to Jorge Almeida for his
comments on a preliminary version of this paper.

\bibliographystyle{amsplain}

\providecommand{\bysame}{\leavevmode\hbox to3em{\hrulefill}\thinspace}
\providecommand{\MR}{\relax\ifhmode\unskip\space\fi MR }
\providecommand{\MRhref}[2]{%
  \href{http://www.ams.org/mathscinet-getitem?mr=#1}{#2}
}
\providecommand{\href}[2]{#2}

%\end{multicols}

\end{document}